 \documentclass[final,5p,times,twocolumn]{elsarticle}

\usepackage{amssymb}

\usepackage{amsthm} 
\usepackage{mathtools}
\usepackage{amsmath}
\usepackage{algorithm}
\usepackage[]{algpseudocode}
\newtheorem{prop}{Proposition}
\usepackage{amsfonts}
\usepackage[table,xcdraw]{xcolor}	
\usepackage{bm}

\usepackage{tkz-graph}
\usetikzlibrary{arrows}
\usepackage{pgfplots}
\usetikzlibrary{shapes.geometric}
\pgfplotsset{compat=1.7}
\usepackage{tikz,caption}
\usepackage{pgfplots}

\usepackage{multirow} 
\usepackage{array}
\newcolumntype{L}[1]{>{\raggedright\arraybackslash}p{#1}}
\usepackage{booktabs}

\begin{document}

\begin{frontmatter}



\title{Robust optimization-based heuristic algorithm \\for the chance-constrained knapsack problem using submodularity}


\author[label1]{Seulgi Joung}
\ead{sgjoung@snu.ac.kr}

\author[label1]{Kyungsik Lee\corref{cor1}}
\ead{optima@snu.ac.kr}

\cortext[cor1]{Corresponding author}
\address[label1]{Department of Industrial Engineering, Seoul National University 1, Gwanak-ro, Gwanak-gu, Seoul 08826, Republic of Korea}

\begin{abstract}
In this paper, we propose a robust optimization-based heuristic algorithm for the chance-constrained binary knapsack problem (CKP). We assume that the weights of items are independent normally distributed. By utilizing the properties of the submodular function, the proposed method approximates the CKP to the robust knapsack problem with a cardinality constrained uncertainty set parameterized by a uncertainty budget parameter. The proposed approach obtains a heuristic solution by solving the approximated robust knapsack problem whose optimal solution can be obtained by solving the ordinary binary knapsack problem iteratively. The computational results show the effectiveness and efficiency of the proposed approach.
\end{abstract}

\begin{keyword}
Chance-constrained knapsack problem \sep Heuristic \sep Robust optimization \sep Submodularity



\end{keyword}

\end{frontmatter}


\section{Introduction}
The binary knapsack problem is a well-known combinatorial optimization problem. There is a set of $n$ items $N=\{1,\dots,n\}$,  each item $j$ having weight ${a}_j$ and profit $p_j$. The objective of the binary knapsack problem is to find a subset of items with the maximum profit sum that satisfies the knapsack capacity $b$. 
The binary knapsack problem is NP-hard, but it can be solved in pseudopolynomial time $O(nb)$ using dynamic programming. Due to its theoretical and practical importance, the binary knapsack problem has been extensively studied over the last few decades (see Martello and Toth \cite{martello1990knapsack} and Kellerer et al. \cite{kellerer2004knapsack}).

When we attempt to solve real-world problems, data uncertainty is inevitable. Two representative approaches to data uncertainty in optimization theory are robust optimization and stochastic optimization. Robust optimization defines an uncertainty set of uncertain data, and stochastic optimization considers the probabilistic characteristics of such data. Chance-constrained programming is a stochastic optimization approach that finds a solution satisfying constraints within a given threshold. In this paper, we consider the binary knapsack problem with uncertain weights of items. To solve this chance-constrained knapsack problem (CKP), the objective of which is to find a subset of items with the maximum profit sum when the probability of satisfying the knapsack constraint is greater than or equal to a given threshold $\rho$, we propose a robust optimization-based heuristic approach. Here we assume that $\rho \geq 0.5$. The CKP can be formulated as follows:
\begin{equation}
\max\left\{ \sum_{j \in N} p_j x_j \mid P \left( \sum_{j \in N}{a}_j x_j \leq b\right) \geq \rho, x \in \mathbb{B}^n \right\}.
\end{equation}
Each binary variable $x_j$ is $1$ if item $j$ is chosen, $0$ otherwise. In general, chance-constrained programming problems are difficult to solve, since the feasible solution set of chance constraints is non-convex in most cases (see Nemirovski and Shapiro \cite{nemirovski2006convex}).    
In this paper, we assume that the weight of each item is independent and normally distributed with mean $\bar{a}_j$ and standard deviation $\sigma_j$. Under this assumption, the chance constraint can be reformulated as a second-order cone constraint; accordingly then, it can be solved using commercial softwares' branch-and-bound methods. However, it remains difficult to derive optimal solutions to large problems. Klopfenstein and Nace \cite{klopfenstein2008robust} proposed a pseudopolynomial-time heuristic algorithm for the CKP only where the bounds on uncertain coefficients are known (e.g. in the case of a uniform distribution). Goyal and Ravi \cite{goyal2010ptas} suggested a polynomial time approximation scheme (PTAS) for the CKP. They reformulated the problem as a parametric LP and provided a rounding algorithm. Han et al. \cite{han2016robust} considered a robust optimization approach for the CKP, specifically an approximation approach that provides an upper bound on the optimal value. 

In the present study, we utilized the properties of the submodular function to solve the CKP. A set function is called a submodular function if it has a diminishing returns property. Submodular functions are used in various discrete optimization problems (see Iwata \cite{iwata2008submodular}). However, only a few studies have employed submodularity for optimization problems with data uncertainty. Atamt{\"u}rk and Narayanan \cite{atamturk2008polymatroids} having considered discrete mean-risk minimization problems, used submodulairy to derive valid inequalities for the binary and mixed-integer problems. Atamt{\"u}rk and Bhardwaj \cite{atamturk2018network} studied a network design problem with uncertain arc capacities, exploiting submodularity and supermodularity to handle an exponential number of probabilistic constraints. 

In this paper, we propose, as an extended version of Klopfenstein and Nace's approach \cite{klopfenstein2008robust}, a heuristic method to find a near-optimal and feasible solution for (\ref{ckp}). Our method can consider a normal distribution by utilizing submodularity to define an interval for the uncertain weights. Using the defined interval, we  then approximate the CKP by a cardinality constrained robust knapsack problem with Bertsimas and Sim model \cite{bertsimas2004price}. The rest of this paper is organized as follows. In Section 2, the robust optimization-based heuristic method is proposed. In Section 3, the results of the proposed approach are reported. In Section 4, our concluding remarks are presented.  
\section{Robust Optimization-Based Heuristic Method}
We assume that each weight $a_j$ of item $j$ has an independent normal distribution with mean $\bar a_j$ and standard deviation $\sigma_j$. 
Then the CKP can be reformulated as
\begin{equation}
\label{ckp}
\begin{aligned}
&\max&& \sum_{j \in N} p_j x_j\\
&\text{ s.t.}& & \sum_{j \in N}\bar{a}_j x_j + \Phi^{-1}(\rho) \sqrt{\sum_{j \in N} \sigma_j^2 x_j^2}  \leq b,\\
&& &x \in \mathbb{B}^n.
\end{aligned}
\end{equation}
In this paper, we assume that $\bar{a}_j \geq 0$, $\sigma_j \geq 0$, and $p_j$ is a non-negative integer for all $j \in N$.
Note that the continuous relaxation of (\ref{ckp}) is a second-order cone programming (SOCP) problem.
The feasible solution set of (\ref{ckp}) is equivalent to the following set, which has a robust constraint with ellipsoidal uncertainty set $\mathcal{U}$:
\begin{displaymath}
\left\{ x \in \mathbb{B}^n \mid \sum_{j \in N}{a}_j x_j \leq b, \forall a \in \mathcal{U} \right\},
\end{displaymath}  
where \begin{displaymath}\mathcal{U} = \left\{\bar{a}_j+ \Phi^{-1}(\rho) \Sigma^{1/2} z \mid \|z\|_2 \leq 1 \right\}, \Sigma = \begin{bmatrix}
       \sigma_1^2 &         &  \\
                  & \ddots  &  \\
                  &         & \sigma_n^2
     \end{bmatrix}.\end{displaymath} 

\noindent The ellipsoidal uncertainty set $\mathcal{U}$ also can be represented as
\begin{equation*}
\label{U}
\mathcal{U} = \left\{\bar{a}_j+ \Sigma^{1/2} \epsilon \mid \sum_{j \in N} \epsilon_j^2 \leq (\Phi^{-1}(\rho))^2 \right\}.
\end{equation*}
We use the properties of the submodular function, which is a set function with a diminishing returns property, to reformulate (\ref{ckp}). Since 
\begin{displaymath}
f(S)=\sum_{j \in S}\bar{a}_j + \Phi^{-1}(\rho) \sqrt{\sum_{j \in S} \sigma_j^2}
\end{displaymath}
for $S \subseteq N$ is a submodular set function (see Atamt{\"u}rk and Narayanan \cite{atamturk2008polymatroids}), the CKP (\ref{ckp}) can be reformulated as
\begin{equation}
\label{reckp}
\begin{aligned}
&\max&& \sum_{j \in N} p_j x_j\\
&\text{ s.t.}& & \sum_{j \in N}\pi_j x_j \leq b, \quad \forall \pi \in ext(\Pi_f),\\
&& &x \in \mathbb{B}^n,
\end{aligned}
\end{equation}
where $\Pi_f = \{\pi \in \mathbb{R}^n \mid \sum_{j \in S}\pi_j \leq f(S), \forall S \subseteq N \}$. Let $ext(\Pi_f)$ be the set of extreme points of $\Pi_f$. The reformulated problem (\ref{reckp}) thus has an exponential number of linear knapsack constraints with the same capacity $b$. The extreme points of $\Pi_f$ can be obtained using the following greedy algorithm (see Edmonds  \cite{edmonds1970submodular}) for all permutations of $N$. For a permutation of $N$ $((1),(2),\dots,(n))$, let $S_j =\{(1),(2),\dots,(j)\}$ and $S_0 = \emptyset$. Then  $\pi_{(j)} = f(S_j) - f(S_{j-1})$ for $j \in N$ is an extreme point of $\Pi_f$.
 
\begin{prop}
$x \in \mathbb{B}^n$ is feasible for (\ref{ckp}) if and only if $x$ is feasible for (\ref{reckp}).
\end{prop}
\begin{proof}
If $x$ is feasible for (\ref{ckp}), then it is feasible for (\ref{reckp}), since $\sum_{j \in N}\pi_j x_j \leq b$ is satisfied for all $ \pi \in ext(\Pi_f)$ by the definition of $\Pi_f$. Now, we will show that if $x^* \in \mathbb{B}^n$ is infeasible for (\ref{ckp}), then $\sum_{j \in N} \pi_j x^*_j > b$ for some $\pi \in ext(\Pi_f)$. We assume that items are sorted in non-increasing order of $x^*_j$. Let $\pi^{\prime}_j = f(S_j)-f(S_{j-1})$. Then, $\sum_{j\in N}\pi^{\prime}_j x_j^* = \sum_{j \in N}\bar{a}_j x^*_j + \Phi^{-1}(\rho) \sqrt{\sum_{j \in N} \sigma_j^2 x_j^{*2}}$ by the definition of $\pi^{\prime}$. Since $x^*$ is infeasible for (\ref{ckp}), we can easily see that $\sum_{j\in N}\pi^{\prime}_j x_j^* > b$.
\end{proof}
We define
 \begin{displaymath}
\underline\pi_j =  \bar{a}_j + \Phi^{-1}(\rho)\left\{\sqrt{\sum_{i \in N} \sigma_i^2} -\sqrt{\sum_{i \in N \setminus \{j\}} \sigma_i^2}\right\} \end{displaymath}
 and
\begin{displaymath}
\bar\pi_j = \bar{a}_j + \Phi^{-1}(\rho) \sigma_j
 \end{displaymath}
 for all $j \in N$. Note that $\underline\pi_j$ is the value of $\pi_j$ when item $j$ is the last item of a permutation of $N$, and $\bar\pi_j$ is the value of $\pi_j$ when item $j$ is the first item of a permutation of $N$. The value of $\pi_j$ is included in an interval $[{\underline\pi}_j, {\bar\pi}_j]$ for all $\pi \in ext(\Pi_f)$, and the knapsack capacities of all constraints of (\ref{reckp}) are the same. Therefore, we can use the cardinality-constrained robust knapsack problem of Bertsimas and Sim \cite{bertsimas2004price} to approximate the CKP. We define the following robust knapsack problem RKP$(\Gamma)$ for a non-negative parameter $\Gamma$ using ${\underline\pi}_j$ and ${\bar\pi}_j$:
\begin{equation}
\label{subrkp}
\begin{aligned}
z^* = &\max&& \sum_{j \in N} p_j x_j\\
&\text{ s.t.}& & \sum_{j \in N}\underline\pi_j x_j +\\
&&&  \max_{\substack{S \cup \{t\}\subseteq N, \\ |S|\leq \lfloor\Gamma\rfloor, \\ t \in N \setminus S}}\left\{\sum_{j \in S}(\bar\pi_j-\underline\pi_j) x_j+(\Gamma - \lfloor\Gamma\rfloor)(\bar\pi_t-\underline\pi_t) x_t\right\} \leq b,\\
&& &x \in \mathbb{B}^n.
\end{aligned}
\end{equation}
Additionally,  the feasible solution set of (\ref{subrkp}) is 
\begin{displaymath}
\left\{ x \in \mathbb{B}^n \mid \sum_{j \in N}{a}_j x_j \leq b, \forall a \in \mathcal{U}(\Gamma) \right\},
\end{displaymath}  
where 
\begin{displaymath}
\mathcal{U}(\Gamma) = \left\{\underline{\pi}_j+ \sum_{j \in N}r_j (\bar{\pi}_j - \underline{\pi}_j) \mid \sum_{j \in N}r_j \leq \Gamma, 0 \leq r_j \leq 1, \forall j \in N \right\}.
\end{displaymath}

We can obtain a feasible solution to the CKP by choosing an appropriate value of $\Gamma$. Figure \ref{fig:uncertaintyset} represents the solution set of the CKP with 2 variables.
The ellipsoidal uncertainty set $\mathcal{U}$ of ${a}$ is the ellipse in Figure \ref{fig:uncertaintyset}. In addition, $\mathcal{U}(0) = \{(\underline{\pi}_1, \underline{\pi}_2)\}$, $\mathcal{U}(1)$ is the dashed triangular area, and $\mathcal{U}(2)$ is the gray area. We can see that $\Gamma=1$ is sufficient to guarantee the feasibility of the CKP (\ref{ckp}) of Figure \ref{fig:uncertaintyset}. 
\begin{figure}[h]
\centering
\begin{tikzpicture}
\usetikzlibrary{patterns}
\draw[] (5, 3) ellipse (3.29 and 1.645);
\draw[dashed] (5, 3) -- (8.29, 3);
\draw[dashed] (5, 3) -- (5, 4.645);
\draw[dashed] (8.29, 3) -- (8.29, 4.645);
\draw[dashed] (5, 4.645) -- (8.29, 4.645);
\draw[] (7.0333, 3.3883) -- (8.29, 3.3883);
\draw[] (7.0333, 3.3883) -- (7.0333, 4.645);
\draw[] (8.29, 3.3883) -- (8.29, 4.645);
\draw[] (7.0333, 4.645) -- (8.29, 4.645);
\draw[] (7.0333, 4.645) -- (8.29, 3.3883);
\draw [draw=black, fill=gray, fill opacity=0.2] (7.0333,3.3883) -- (8.29,3.3883) -- (8.29,4.645) -- (7.0333,4.645) -- cycle;
\pattern[pattern=north east lines] (7.0333,3.3883) -- (8.29,3.3883) -- (7.0333,4.645)--cycle;
\draw (5,3) node[left] {\small$(\bar{a}_1, \bar{a}_2)$};
\draw[] (8.290,4.645) node[right] {\small$(\bar{\pi}_1, \bar{\pi}_2)$};
\draw[] (7.0333,3.3883) node[left] {\small$(\underline{\pi}_1, \underline{\pi}_2)$};
\draw [color=black, fill=black] (5,3) circle (.03);
\draw [color=black, fill=black] (8.290,4.645) circle (.03);
\draw [color=black, fill=black] (7.0333,3.3883) circle (.03);
\end{tikzpicture}
\caption{Uncertainty set of ${a}$}
\label{fig:uncertaintyset}
\end{figure}
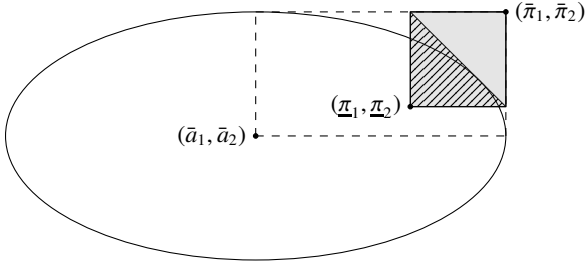

We increase the $\Gamma$ value from $0$ until the optimal solution to RKP$(\Gamma)$ is feasible for (\ref{ckp}). Since an optimal solution to RKP$(n)$ is feasible for (\ref{reckp}), we can find a feasible solution for it using the proposed approach.

\begin{prop}
\label{prop2}
An optimal solution to RKP$(n)$ is feasible for (\ref{reckp}).
\end{prop}
\begin{proof}
Since $\pi_j \leq \bar\pi_j$ for all $j \in N$ and $\pi \in \Pi_f$, a feasible solution for RKP$(n)$ is also feasible for (\ref{reckp}).
\end{proof}
 Each robust knapsack problem (\ref{subrkp}) can be solved by solving $n+1$ ordinary knapsack problems as follows (see Bertsimas and Sim \cite{bertsimas2003robust}). Assume that items are ordered in non-increasing order of $(\bar\pi_j-\underline\pi_j)$. We also define $(\bar \pi_{n+1} - \underline \pi_{n+1}) = 0$. Then, (\ref{subrkp}) can be solved by solving $n+1$ ordinary knapsack problems $g^l$ for $l = 1,\dots, n+1$:
\begin{displaymath}
z^* = \min_{l = 1, \dots, n+1} g^l,
\end{displaymath}  
where 

\begin{equation}
\label{sub_rkp_kp}
\begin{aligned}
g^l = &\max&& \sum_{j \in N} p_j x_j\\
&\text{ s.t.}& & \sum_{j \in N}\underline\pi_j x_j + \sum_{j=1}^{l}\left\{(\bar\pi_j-\underline\pi_j) - (\bar\pi_l-\underline\pi_l)\right\}x_j \\ 
&&&\leq b - \Gamma (\bar\pi_l-\underline\pi_l),\\
&& &x \in \mathbb{B}^n.
\end{aligned}
\end{equation}

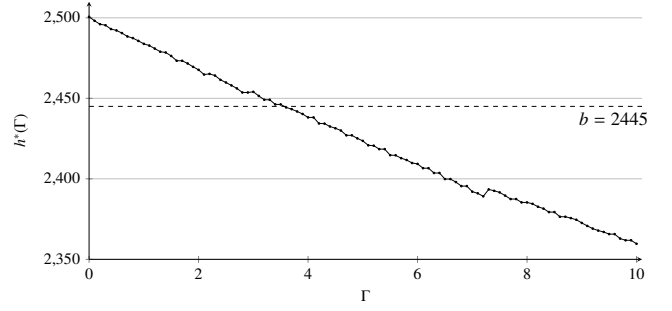
\begin{figure}[h]
	\centering
		\begin{tikzpicture}[scale = 0.6]
\begin{axis}[
  ymin=2350,
  ymax=2510,
  xmin=0,
  xmax=10.1,
  x label style={at={(axis description cs:0.5,-0.1)},anchor=north},
  y label style={at={(axis description cs:-0.1,0.5)},anchor=south},
  ytick={2350, 2400, 2450, 2500},
  xtick={0,2,4,6,8,10},
  xlabel=$\Gamma$,
  ylabel=$h^*(\Gamma)$,
  x=1.2cm,
  axis lines=left,
  ymajorgrids,
  ]
\addplot[color=black,mark=*, mark size = 0.5pt] coordinates {
		 (	0	,	2500.64	)
(	0.1	,	2498.16	)
(	0.2	,	2496.02	)
(	0.3	,	2495.4	)
(	0.4	,	2493.03	)
(	0.5	,	2492.17	)
(	0.6	,	2490.53	)
(	0.7	,	2488.43	)
(	0.8	,	2487.4	)
(	0.9	,	2485.68	)
(	1	,	2483.85	)
(	1.1	,	2482.78	)
(	1.2	,	2480.91	)
(	1.3	,	2478.95	)
(	1.4	,	2478.48	)
(	1.5	,	2476.32	)
(	1.6	,	2473.34	)
(	1.7	,	2473.34	)
(	1.8	,	2471.65	)
(	1.9	,	2469.57	)
(	2	,	2467.7	)
(	2.1	,	2464.78	)
(	2.2	,	2465.18	)
(	2.3	,	2464.15	)
(	2.4	,	2461.49	)
(	2.5	,	2459.84	)
(	2.6	,	2458.03	)
(	2.7	,	2456.16	)
(	2.8	,	2453.63	)
(	2.9	,	2453.63	)
(	3	,	2454.05	)
(	3.1	,	2451.49	)
(	3.2	,	2449.12	)
(	3.3	,	2449.12	)
(	3.4	,	2446.22	)
(	3.5	,	2446.22	)
(	3.6	,	2444.35	)
(	3.7	,	2443.27	)
(	3.8	,	2441.84	)
(	3.9	,	2440.29	)
(	4	,	2438.14	)
(	4.1	,	2438.14	)
(	4.2	,	2434.38	)
(	4.3	,	2434.3	)
(	4.4	,	2432.51	)
(	4.5	,	2431.4	)
(	4.6	,	2430	)
(	4.7	,	2427.02	)
(	4.8	,	2427.02	)
(	4.9	,	2425.13	)
(	5	,	2423.58	)
(	5.1	,	2420.82	)
(	5.2	,	2420.6	)
(	5.3	,	2418.45	)
(	5.4	,	2418.45	)
(	5.5	,	2414.69	)
(	5.6	,	2414.61	)
(	5.7	,	2412.83	)
(	5.8	,	2411.71	)
(	5.9	,	2409.84	)
(	6	,	2409.28	)
(	6.1	,	2406.61	)
(	6.2	,	2406.61	)
(	6.3	,	2403.63	)
(	6.4	,	2403.63	)
(	6.5	,	2399.87	)
(	6.6	,	2399.87	)
(	6.7	,	2398	)
(	6.8	,	2395.49	)
(	6.9	,	2395.49	)
(	7	,	2392.03	)
(	7.1	,	2391	)
(	7.2	,	2389.13	)
(	7.3	,	2393.43	)
(	7.4	,	2392.57	)
(	7.5	,	2391.57	)
(	7.6	,	2389.59	)
(	7.7	,	2387.44	)
(	7.8	,	2387.44	)
(	7.9	,	2385.35	)
(	8	,	2385.35	)
(	8.1	,	2384.49	)
(	8.2	,	2382.68	)
(	8.3	,	2381.51	)
(	8.4	,	2379.35	)
(	8.5	,	2379.35	)
(	8.6	,	2376.46	)
(	8.7	,	2376.46	)
(	8.8	,	2375.6	)
(	8.9	,	2374.6	)
(	9	,	2372.62	)
(	9.1	,	2370.76	)
(	9.2	,	2369.1	)
(	9.3	,	2367.86	)
(	9.4	,	2367	)
(	9.5	,	2365.67	)
(	9.6	,	2365.67	)
(	9.7	,	2362.91	)
(	9.8	,	2361.83	)
(	9.9	,	2361.83	)
(	10	,	2359.67	)
};
\addplot[color=black, dashed] coordinates {		(	0	,	2445	)		(	10.1	,	2445	)	};
\end{axis}
\node at (11.5, 3.1) {\scriptsize{$b=2445$}};
\end{tikzpicture}
\caption{Change of $h^*$ values for different $\Gamma$ values}
\label{fig2}
\end{figure}

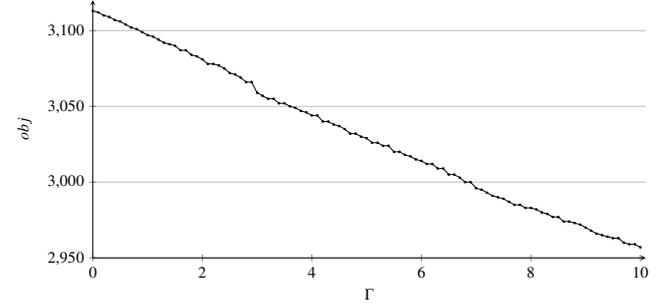
\begin{figure}[t]
	\centering
		\begin{tikzpicture}[scale = 0.6]
\begin{axis}[
  ymin=2950,
  ymax=3120,
  xmin=0,
  xmax=10.1,
  x label style={at={(axis description cs:0.5,-0.1)},anchor=north},
  y label style={at={(axis description cs:-0.1,0.5)},anchor=south},
  ytick={2950, 3000, 3050, 3100},
  xtick={0,2,4,6,8,10},
  xlabel=$\Gamma$,
  ylabel=$obj$,
  x=1.2cm,
  axis lines=left,
  ymajorgrids,
  ]
\addplot[color=black,mark=*, mark size = 0.5pt] coordinates {
(	0	,	3113	)
(	0.1	,	3112	)
(	0.2	,	3110	)
(	0.3	,	3109	)
(	0.4	,	3107	)
(	0.5	,	3106	)
(	0.6	,	3104	)
(	0.7	,	3102	)
(	0.8	,	3101	)
(	0.9	,	3099	)
(	1	,	3097	)
(	1.1	,	3096	)
(	1.2	,	3094	)
(	1.3	,	3092	)
(	1.4	,	3091	)
(	1.5	,	3090	)
(	1.6	,	3087	)
(	1.7	,	3087	)
(	1.8	,	3084	)
(	1.9	,	3083	)
(	2	,	3081	)
(	2.1	,	3078	)
(	2.2	,	3078	)
(	2.3	,	3077	)
(	2.4	,	3075	)
(	2.5	,	3072	)
(	2.6	,	3071	)
(	2.7	,	3069	)
(	2.8	,	3066	)
(	2.9	,	3066	)
(	3	,	3059	)
(	3.1	,	3057	)
(	3.2	,	3055	)
(	3.3	,	3055	)
(	3.4	,	3052	)
(	3.5	,	3052	)
(	3.6	,	3050	)
(	3.7	,	3049	)
(	3.8	,	3047	)
(	3.9	,	3046	)
(	4	,	3044	)
(	4.1	,	3044	)
(	4.2	,	3040	)
(	4.3	,	3040	)
(	4.4	,	3038	)
(	4.5	,	3037	)
(	4.6	,	3035	)
(	4.7	,	3032	)
(	4.8	,	3032	)
(	4.9	,	3030	)
(	5	,	3029	)
(	5.1	,	3026	)
(	5.2	,	3026	)
(	5.3	,	3024	)
(	5.4	,	3024	)
(	5.5	,	3020	)
(	5.6	,	3020	)
(	5.7	,	3018	)
(	5.8	,	3017	)
(	5.9	,	3015	)
(	6	,	3014	)
(	6.1	,	3012	)
(	6.2	,	3012	)
(	6.3	,	3009	)
(	6.4	,	3009	)
(	6.5	,	3005	)
(	6.6	,	3005	)
(	6.7	,	3003	)
(	6.8	,	3000	)
(	6.9	,	3000	)
(	7	,	2996	)
(	7.1	,	2995	)
(	7.2	,	2993	)
(	7.3	,	2991	)
(	7.4	,	2990	)
(	7.5	,	2989	)
(	7.6	,	2987	)
(	7.7	,	2985	)
(	7.8	,	2985	)
(	7.9	,	2983	)
(	8	,	2983	)
(	8.1	,	2982	)
(	8.2	,	2980	)
(	8.3	,	2979	)
(	8.4	,	2977	)
(	8.5	,	2977	)
(	8.6	,	2974	)
(	8.7	,	2974	)
(	8.8	,	2973	)
(	8.9	,	2972	)
(	9	,	2970	)
(	9.1	,	2968	)
(	9.2	,	2966	)
(	9.3	,	2965	)
(	9.4	,	2964	)
(	9.5	,	2963	)
(	9.6	,	2963	)
(	9.7	,	2960	)
(	9.8	,	2959	)
(	9.9	,	2959	)
(	10	,	2957	)

};\end{axis}
\end{tikzpicture}
\caption{Change of optimal values of RKP$(\Gamma)$ for different $\Gamma$ values}
\label{fig3}
\end{figure}

Note that when we decompose a robust knapsack problem into $n+1$ ordinary knapsack problems, the integrality of $\Gamma$ is not a  necessary condition. Klopfenstein and Nace \cite{klopfenstein2008robust} considered only integer $\Gamma$ values, but we also consider non-integer $\Gamma$ values. The ordinary knapsack problem is NP-hard, but (\ref{sub_rkp_kp}) can be solved in $O(nU)$, where $U$ is an upper bound on the optimal objective value, using dynamic programming when the profit values are integers (see Kellerer et al. \cite{kellerer2004knapsack}).

In the proposed heuristic method, choosing an appropriate value of $\Gamma$ is essential. We define 
\begin{displaymath}
h^*(\Gamma) = \sum_{j \in N}\bar{a}_j x^*_j + \Phi^{-1}(\rho) \sqrt{\sum_{j \in N} \sigma_j^2 x_j^{*2}},
\end{displaymath} 
where $x^*$ is an optimal solution to RKP$(\Gamma)$. We compare the $h^*(\Gamma)$ value and the optimal value of RKP$(\Gamma)$ for different $\Gamma$ values. 
Figure \ref{fig2} shows the change of $h^*(\Gamma)$ for an SC-type CKP instance (see Han et al. \cite{han2016robust}) when $n=100$. As can be seen, $h^*(\Gamma)$ tends to decrease as $\Gamma$ increases; however, $h^*(\Gamma)$ does not decrease monotonously. If $h^*(\Gamma) \leq b$, the optimal solution $x^*$ is feasible for CKP (\ref{ckp}). Therefore, we cannot guarantee that there exists a unique value of $\Gamma^*$ such that the optimal solution to RKP$(\Gamma)$ is feasible for (\ref{ckp}) if and only if $\Gamma \geq \Gamma^*$. Additionally, Figure \ref{fig3} represents the change of the optimal values of RKP$(\Gamma)$ for the same instance. Naturally, the optimal value of RKP$(\Gamma)$ decreases as $\Gamma$ increases. Therefore, in order to find a small value of $\Gamma$ such that the optimal solution to RKP$(\Gamma)$ is feasible for (\ref{ckp}), we use the following jump search method. Figure \ref{fig4} shows the idea of the jump search. 
\begin{figure}[ht]
\centering
\tikzset{every picture/.style={line width=0.75pt}} 

\begin{tikzpicture}[x=0.75pt,y=0.75pt,yscale=-1,xscale=1,scale=0.9]

\draw    (105.9,125.05) .. controls (132.04,75.1) and (167.92,73.38) .. (195.81,119.92) ;
\draw [shift={(196.65,121.35)}, rotate = 239.74] [fill={rgb, 255:red, 0; green, 0; blue, 0 }  ][line width=0.75]  [draw opacity=0] (10.72,-5.15) -- (0,0) -- (10.72,5.15) -- (7.12,0) -- cycle    ;

\draw [line width=1.5]    (87,125) -- (391.3,125) ;
\draw [shift={(398.3,125)}, rotate = 539.9300000000001] [color={rgb, 255:red, 0; green, 0; blue, 0 }  ][line width=1.5]    (14.21,-4.28) .. controls (9.04,-1.82) and (4.3,-0.39) .. (0,0) .. controls (4.3,0.39) and (9.04,1.82) .. (14.21,4.28)   ;

\draw  [dash pattern={on 0.84pt off 2.51pt}]  (283.3,124.6) .. controls (293.28,104.97) and (306.54,105.29) .. (312.89,118.62) ;
\draw [shift={(313.65,120.35)}, rotate = 248.2] [fill={rgb, 255:red, 0; green, 0; blue, 0 }  ][line width=0.75]  [draw opacity=0] (10.72,-5.15) -- (0,0) -- (10.72,5.15) -- (7.12,0) -- cycle    ;

\draw    (195.3,124.6) .. controls (221.44,74.65) and (257.91,72.39) .. (285.81,118.92) ;
\draw [shift={(286.65,120.35)}, rotate = 239.74] [fill={rgb, 255:red, 0; green, 0; blue, 0 }  ][line width=0.75]  [draw opacity=0] (10.72,-5.15) -- (0,0) -- (10.72,5.15) -- (7.12,0) -- cycle    ;

\draw    (283.3,124.6) .. controls (309.44,74.65) and (342.97,72.39) .. (370.81,118.92) ;
\draw [shift={(371.65,120.35)}, rotate = 239.74] [fill={rgb, 255:red, 0; green, 0; blue, 0 }  ][line width=0.75]  [draw opacity=0] (10.72,-5.15) -- (0,0) -- (10.72,5.15) -- (7.12,0) -- cycle    ;

\draw    (372.7,124.15) .. controls (355.39,149.47) and (325.55,186.73) .. (287.23,128.53) ;
\draw [shift={(286.65,127.65)}, rotate = 417] [fill={rgb, 255:red, 0; green, 0; blue, 0 }  ][line width=0.75]  [draw opacity=0] (10.72,-5.15) -- (0,0) -- (10.72,5.15) -- (7.12,0) -- cycle    ;

\draw  [dash pattern={on 0.84pt off 2.51pt}]  (314.3,124.6) .. controls (324.28,104.97) and (337.54,105.29) .. (343.89,118.62) ;
\draw [shift={(344.65,120.35)}, rotate = 248.2] [fill={rgb, 255:red, 0; green, 0; blue, 0 }  ][line width=0.75]  [draw opacity=0] (10.72,-5.15) -- (0,0) -- (10.72,5.15) -- (7.12,0) -- cycle    ;

\draw  [fill={rgb, 255:red, 255; green, 255; blue, 255 }  ,fill opacity=1 ]  (371.65, 125) circle [x radius= 3.65, y radius= 3.65]  ;
\draw  [fill={rgb, 255:red, 155; green, 155; blue, 155 }  ,fill opacity=1 ]  (286.65, 125) circle [x radius= 3.65, y radius= 3.65]  ;
\draw  [fill={rgb, 255:red, 155; green, 155; blue, 155 }  ,fill opacity=1 ]  (313.65, 125) circle [x radius= 3.65, y radius= 3.65]  ;
\draw  [fill={rgb, 255:red, 255; green, 255; blue, 255 }  ,fill opacity=1 ]  (344.65, 125) circle [x radius= 3.65, y radius= 3.65]  ;
\draw  [fill={rgb, 255:red, 155; green, 155; blue, 155 }  ,fill opacity=1 ]  (196.65, 125) circle [x radius= 3.65, y radius= 3.65]  ;
\draw  [fill={rgb, 255:red, 155; green, 155; blue, 155 }  ,fill opacity=1 ]  (107.65, 125) circle [x radius= 3.65, y radius= 3.65]  ;
\draw (152,77) node [scale=0.7] [align=left] {step 1 ($+m_1$)};
\draw (244,77) node [scale=0.7] [align=left] {step 2 ($+m_1$)};
\draw (329,77) node [scale=0.7] [align=left] {step 3 ($+m_1$)};
\draw (331,170) node [scale=0.7] [align=left] {step 4 ($-m_1$)};
\draw (297,98) node [scale=0.7] [align=left] {step 5\\($+m_2$)};
\draw (333,98) node [scale=0.7] [align=left] {step 6\\($+m_2$)};
\draw (415,125) node [scale=0.7]  {$\Gamma$};
\draw (108,138) node [scale=0.7] [align=left] {0};
\end{tikzpicture}
\caption{Jump search algorithm}
\label{fig4}
\end{figure}
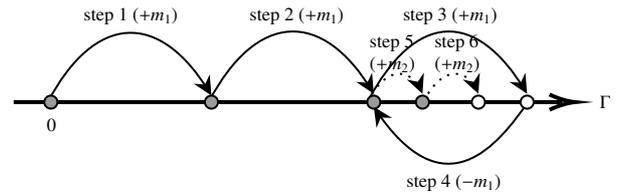
Here the gray points mean that the optimal solution to RKP$(\Gamma)$ is infeasible for (\ref{ckp}), while the white points mean that the optimal solution to RKP$(\Gamma)$ is feasible for (\ref{ckp}). Let $m_k$ be the $k$th step size for $k = 1 \dots K$, where $K$ is a predefined number of iterations. We define $m_1 = n \times (1/u)$ and $m_{k+1} = m_k \times (1/v), k = 1, \dots , K-1$ for positive integers $u$ and $v$. It guarantees the feasibility of the obtained solution by the Proposition \ref{prop2}. We find the smallest nonnegative integer $d_1$ such that $h^*(m_1 d_1) \leq b$. Here, we solve RKP$(m_1 d)$ for $d= 0, 1, 2, \dots$ until $h^*(m_1 d) \leq b$. In Figure \ref{fig4}, we can see that $d_1 = 3$. Here, the $\Gamma$ value is updated to $m_1 d_1$. Then, we start again from $m(d_1 -1)$. Let $m_2$ be the second step size that is smaller than $m_1$. We find the smallest non-negative integer $d_2$ value such that $h^*(m_1(d_1-1) + m_2 d_2) \leq b$. We repeat this procedure $K$ times to find an appropriate $\Gamma$ value. In our test, we use $u = n$ and $v = 10$.  The detailed jump search algorithm is given as Algorithm 1. 
\begin{algorithm}[H]
\caption{Jump search algorithm to find $\Gamma$ value}
\label{alg1}
\begin{algorithmic}[1]
\State Input: $K, u, v$
\State $\gamma = 0$, $k=1$, $m_1 = n \times (1/u)$
\While{$\gamma \leq n$}  
\State Solve RKP($\gamma$)
\If{the optimal solution of RKP($\gamma$) is feasible to (\ref{subrkp})}
\If{$k == K$} 
\State $\Gamma = \gamma$
\State Stop
\Else
\State $\gamma = \gamma - m_k$
\State $m_{k+1} = m_{k}\times (1/v)$
\State $k = k + 1$
\EndIf
\EndIf
\State $\gamma = \gamma + m_k$
\EndWhile
\State Output: $\Gamma$, the optimal value of RKP($\Gamma$)
\end{algorithmic}
\end{algorithm}

\section{Computational Results}
In this section, we report the computational results of the proposed heuristic method. We compared the results with the robust optimization-based approximation approach of Han et al. \cite{han2016robust} and a commercial software (CPLEX 12.7). All of the tests were performed on a computer Intel(R) Core(TM) i7-4770S processor with 3.10 GHz and 16GB RAM. The proposed method and the approximation approach of Han et al. \cite{han2016robust} were implemented using C++. 

We used the same instances as in Han et al. \cite{han2016robust}, as derived from the robust knapsack problem instances of Monaci et al. \cite{monaci2013exact}. The instance types used here are as follows:
\begin{itemize}
\item SC (strongly correlated) instances: $\bar a_j$ is an integer randomly generated in [1,100] and $p_j = \bar a_j + 10$.
\item IC (inverse strongly correlated) instances: $p_j$ is a randomly generated integer in [1,100] and $\bar a_j = \min \{100, p_j + 10\}$. 
\item SS (subset sum) instances: $\bar a_j $ is randomly chosen in [1,100] and $p_j = \bar a_j$.
\end{itemize}
The standard deviation value $\sigma_j$ was randomly generated in $[0.1 \bar a_j, 0.2 \bar a_j]$, and the knapsack capacity $b = \lfloor \sum_{j \in N} \bar a_j\rfloor$. 
We solved each ordinary knapsack problem (\ref{sub_rkp_kp}) using the minimal algorithm coded in C by Pisinger \cite{pisinger2000minimal}. We used the same approach of Han et al. \cite{han2016robust} to transform the non-integral data to integer data. Each weight value was rounded up after multiplying by $10^6$, and the knapsack capacity was rounded down after multiplying by $10^6$.

\begin{table}[h]
\small
\centering
\begin{tabular}{L{0.75cm}L{0.5cm}L{0.9cm}L{0.9cm}L{0.9cm}L{0.9cm}L{0.9cm}}
\toprule
\multicolumn{2}{l}{Instances} & \multicolumn{5}{l}{$K$} \\
\cmidrule(l){1-2} \cmidrule(l){3-7}
\multicolumn{1}{l}{$\rho$} & \multicolumn{1}{l}{type} & \multicolumn{1}{l}{1} & \multicolumn{1}{l}{2} & \multicolumn{1}{l}{3} & \multicolumn{1}{l}{4} & \multicolumn{1}{l}{5} \\\midrule
\multirow{3}{*}{0.85} & SC & 3131.7 & 3135.5 & 3135.7 & 3135.7 & 3135.7 \\
 & IC & 2637.8 & 2640.8 & 2641.0 & 2641.0 & 2641.0 \\
 & SS & 2484.0 & 2486.8 & 2486.8 & 2486.9 & 2486.9 \\\midrule
\multirow{3}{*}{0.90} & SC & 3117.6 & 3122.4 & 3122.9 & 3122.9 & 3122.9 \\
 & IC & 2621.3 & 2625.2 & 2625.6 & 2625.6 & 2625.6 \\
 & SS & 2472.0 & 2475.0 & 2475.6 & 2475.6 & 2475.6 \\\midrule
\multirow{3}{*}{0.95} & SC & 3096.2 & 3102.9 & 3103.3 & 3103.3 & 3103.3 \\
 & IC & 2595.5 & 2600.1 & 2601.1 & 2601.1 & 2601.1 \\
 & SS & 2454.3 & 2458.6 & 2458.8 & 2458.8 & 2458.8 \\\bottomrule
\end{tabular}
\caption{Change of objective values for different $K$ values when $n=100$}
\label{table1}
\end{table}

\begin{table}[h]
\small
\centering
\begin{tabular}{L{0.75cm}L{0.5cm}L{0.9cm}L{0.9cm}L{0.9cm}L{0.9cm}L{0.9cm}}
\toprule
\multicolumn{2}{l}{Instances} & \multicolumn{5}{l}{$K$} \\
\cmidrule(l){1-2} \cmidrule(l){3-7}
\multicolumn{1}{l}{$\rho$} & \multicolumn{1}{l}{type} & \multicolumn{1}{l}{1} & \multicolumn{1}{l}{2} & \multicolumn{1}{l}{3} & \multicolumn{1}{l}{4} & \multicolumn{1}{l}{5} \\\midrule
\multirow{3}{*}{0.85} & SC & 515 & 1141 & 1727 & 2333 & 2899 \\
 & IC & 414 & 1172 & 1757 & 2424 & 2899 \\
 & SS & 515 & 1202 & 1646 & 2313 & 2868 \\\midrule
\multirow{3}{*}{0.90} & SC & 515 & 1141 & 1535 & 2091 & 2535 \\
 & IC & 414 & 1192 & 1828 & 2374 & 2969 \\
 & SS & 515 & 1212 & 1707 & 2343 & 2909 \\\midrule
\multirow{3}{*}{0.95} & SC & 515 & 1111 & 1596 & 2091 & 2697 \\
 & IC & 414 & 1222 & 1788 & 2404 & 2929 \\
 & SS & 515 & 1202 & 1889 & 2353 & 2838\\\bottomrule
\end{tabular}
\caption{Change of numbers of solved knapsack problems for different $K$ values when $n=100$}
\label{table2}
\end{table}

\begin{table*}[ht]
\small
\centering
\begin{tabular}{ccccccccccccc}
\toprule
\multirow{2}{*}{$n$} & \multirow{2}{*}{$\rho$} & \multirow{2}{*}{type} & \multicolumn{4}{c}{\begin{tabular}{@{}c@{}}Proposed heuristic algorithm\\ ($K=3$)\end{tabular}} & \multicolumn{3}{c}{\begin{tabular}{@{}c@{}}Approximation approach \\of Han et al. \cite{han2016robust}\end{tabular}} & \multicolumn{3}{c}{CPLEX} \\
\cmidrule(l){4-7}\cmidrule(l){8-10}\cmidrule(l){11-13}
 &  &  & obj & time & $\Gamma$ & \# knapsack & obj & time & \# knapsack & obj & ub & time \\\midrule
\multirow{9}{*}{100} & \multirow{3}{*}{0.85} & SC & 3135.7 & 0.2 & 3.7 & 1727 & 3136.0 & 5.9 & \multirow{3}{*}{39601} & 3135.8 &  3135.8 & 38.5 \\
 &  & IC & 2641.0 & 0.1 & 2.8 & 1757 & 2641.0 & 4.6 &  & 2641.0 & 2641.0 & 47.4 \\
 &  & SS & 2486.8 & 0.1 & 3.7 & 1646 & 2487.3 & 4.2 &  & 2487.2$^*$ & 2491.0 & 600.0(10) \\\cmidrule(l){2-13}
 & \multirow{3}{*}{0.90} & SC & 3122.9 & 0.1 & 3.7 & 1535 & 3123.2 & 5.5 & \multirow{3}{*}{39601} & 3122.9  & 3122.9 & 26.2 \\
 &  & IC & 2625.6 & 0.1 & 2.8 & 1828 & 2625.9 & 4.4 &  & 2625.8 & 2625.8 & 3.9 \\
 &  & SS & 2475.6 & 0.1 & 3.7 & 1707 & 2476.3 & 4.4 &  & 2475.9$^*$ & 2481.8 & 600.0(10) \\\cmidrule(l){2-13}
 & \multirow{3}{*}{0.95} & SC & 3103.3 & 0.2 & 3.6 & 1596 & 3103.7 & 5.3 & \multirow{3}{*}{39601} & 3103.6 & 3103.6 & 5.4 \\
 &  & IC & 2601.1 & 0.1 & 2.9 & 1788 & 2601.3 & 4.4 &  & 2601.2 & 2601.2 & 20.0 \\
 &  & SS & 2458.8 & 0.1 & 3.7 & 1889 & 2459.7 & 4.3 &  & 2459.4$^*$ & 2465.1 & 600.0(10) \\\midrule
\multirow{9}{*}{500} & \multirow{3}{*}{0.85} & SC & 15943.5 & 3.4 & 7.2 & 9719 &  \multicolumn{2}{c}{-}  & \multirow{3}{*}{998001} & 15943.4$^*$ & 15947.2 & 489.7(8) \\
 &  & IC & 13450.6 & 1.5 & 5.7 & 10321 &  \multicolumn{2}{c}{-}  &  & 13450.2$^*$ & 13453.7 & 487.4(8) \\
 &  & SS & 12457.4 & 0.9 & 8.1 & 9770 &  \multicolumn{2}{c}{-}  &  & 12458.7$^*$ & 12479.8 & 600.0(10) \\\cmidrule(l){2-13}
 & \multirow{3}{*}{0.90} & SC & 15911.2 & 3.1 & 7.3 & 9669 &  \multicolumn{2}{c}{-}  & \multirow{3}{*}{998001} & 15913.2$^*$ & 15915.7 & 504.5(8) \\
 &  & IC & 13413.9 & 1.3 & 5.7 & 10070 &  \multicolumn{2}{c}{-}  &  & 13413.7$^*$ & 13416.3 & 337.7(5) \\
 &  & SS & 12431.4 & 0.8 & 8.1 & 9369 &  \multicolumn{2}{c}{-}  &  & 12433.1$^*$ & 12459.2 & 600.0(10) \\\cmidrule(l){2-13}
 & \multirow{3}{*}{0.95} & SC & 15863.8 & 2.7 & 7.3 & 9269 &  \multicolumn{2}{c}{-}  & \multirow{3}{*}{998001} & 15865.0$^*$ & 15869.3 & 600.0(10) \\
 &  & IC & 13358.8 & 1.2 & 5.8 & 10271 &  \multicolumn{2}{c}{-}  &  & 13358.7$^*$ & 13361.2 & 406.3(6) \\
 &  & SS & 12392.8 & 0.8 & 8.1 & 8768 &  \multicolumn{2}{c}{-}  &  & 12395.1$^*$ & 12427.5 & 600.0(10) \\\midrule
\multirow{9}{*}{1000} & \multirow{3}{*}{0.85} & SC & 32115.8 & 13.9 & 9.9 & 21321 &  \multicolumn{2}{c}{-}  & \multirow{3}{*}{3996001} & 32116.0$^*$ & 32121.1 & 429.7(7) \\
 &  & IC & 27154.2 & 4.4 & 8.0 & 17918 &  \multicolumn{2}{c}{-}  & & 27153.1$^*$ & 27157.5 & 305.3(5) \\
 &  & SS & 25155.6 & 3.3 & 11.3 & 21622 &  \multicolumn{2}{c}{-}  &  & 25157.0$^*$ & 25188.9 & 600.0(10) \\\cmidrule(l){2-13}
 & \multirow{3}{*}{0.90} & SC & 32070.6 & 13.2 & 9.9 & 23423 &  \multicolumn{2}{c}{-}  & \multirow{3}{*}{3996001} & 32072.4$^*$ &  32076.5& 312.3(4) \\
 &  & IC & 27099.5 & 4.7 & 8.0 & 19520 &  \multicolumn{2}{c}{-}  &  & 27098.3$^*$ & 27103.4 & 491.8(8) \\
 &  & SS & 25118.2 & 3.4 & 11.3 & 22022 &  \multicolumn{2}{c}{-}  &  & 25120.7$^*$ & 25159.5 & 600.0(10) \\\cmidrule(l){2-13}
 & \multirow{3}{*}{0.95} & SC & 31999.8 & 11.6 & 10.2 & 22122 &  \multicolumn{2}{c}{-}  & \multirow{3}{*}{3996001} & 32005.1$^*$ & 32010.0 & 600.0(10) \\
 &  & IC & 27020.2 & 4.2 & 8.0 & 19520 &  \multicolumn{2}{c}{-} & & 27019.5$^*$ & 27024.2 & 472.6(7)   \\
 &  & SS & 25063.4 & 3.3 & 11.3 & 21822 &  \multicolumn{2}{c}{-} & & 25066.4$^*$ & 25116.3 & 600.0(10)  \\\bottomrule
\end{tabular}
\caption{Computational results for the chance-constrained knapsack problem (\ref{ckp})}
\label{table3}
\end{table*}

Tables \ref{table1} and \ref{table2} show the changes in the objective values and numbers of solved knapsack problems for different $K$ values when $n=100$. Naturally, the heuristic algorithm finds a better solution as the number of iterations $K$ increases. As can be seen in Table \ref{table1}, the objective value of the obtained heuristic solution increases as $K$ increases. Additionally, the number of solved ordinary knapsack problems increases as the value of $K$ increases. When $K$ is equal to or greater than 3, there is almost no difference in the objective values. Therefore, we concluded that $K=3$ is sufficient to obtain qualified feasible solutions. In obtaining the following computational results, we also used $K=3$. 

We compared the proposed heuristic algorithm with the approximation approach of Han et al. \cite{han2016robust} and CPLEX 12.7. The approximation approach of Han et al. \cite{han2016robust} obtains the upper bound on the optimal objective value of (\ref{ckp}). When we tested their approximation method, we set the number of segments to $m = 4n$, as recommended in their paper. Note that their approach does not guarantee the feasibility of the solution, since it finds an upper bound on the optimal objective value. Also, their method solves the ordinary knapsack problem $mn-m+1$ times to obtain the approximate objective value. CPLEX, meanwhile, solved the reformulated CKP (\ref{ckp}) using the default settings. We reported the average results of 10 instances for each combination of instance types (SC, IC, SS), $n=100, 500, 1000$ and $\rho = 0.85, 0.90, 0.95$. The time limit was 600 seconds. The results are summarized in Table \ref{table3}. The average objective value (obj), the average computation time in seconds (time), the average value of $\Gamma$ obtained by the jump search algorithm of our heuristic method ($\Gamma$), the average number of solved ordinary knapsack problems (\# knapsack), and the best upper bound obtained by CPLEX within the time limit (ub) are given. As for the CPLEX results, when some instances could not obtain the optimal solution within the time limit, the average objective value of the best feasible solution was given. For this case, the objective value was marked with an asterisk, and the time limit of 600 seconds was applied to the calculation of the average computation time. Also, the number of instances not solved to optimality were given in parentheses.

As can be seen, the proposed heuristic algorithm solved all instances within the time limit. However, the approximation approach of Han et al. \cite{han2016robust} and CPLEX could not solve large-size instances $n=500$ and $1000$ within the time limit. Naturally, the computation time and the number of solved knapsack problems increase as the problem size increases. The average computation time varies according to the instance type, and the difference is greater in CPLEX than in other methods. Whereas we can see that the value of $\Gamma$ also increases, it does not increase in proportion to the problem size. The obtained objective values using the proposed method are comparable to the objective values obtained using CPLEX. Furthermore, for some large-size instances, the objective values obtained using our heuristic method are better than those obtained using CPLEX within the time limit. This means that our proposed method can be effective in practice, since it obtains good feasible solutions within a short time. The approximation approach of Han et al. \cite{han2016robust} needs to solve more knapsack problems than our heurstic method. When $n=1000$, the approach of Han et al. \cite{han2016robust} has to solve the knapsack problem almost 200 times more than our heuristic method. Also, our method can obtain feasible solutions, while the method of Han et al. \cite{han2016robust} cannot. 

\section{Conclusion}
In this paper, we propose a heuristic algorithm for the CKP wherein the weight of each item has an independent normal distribution. The problem can be reformulated as an integer SOCP; however, it is difficult to obtain an optimal solution in a short time using commercial softwares. Therefore we propose a heuristic algorithm using robust optimization and submodularity. The proposed algorithm solves ordinary knapsack problems iteratively. The computational results show the effectiveness and efficiency of our method. The basic idea of the proposed approach is to define a closed interval for each uncertain weight using submodularity and to approximate the CKP to a robust knapsack problem which can be solved relatively easier. We expect that it can be applied to other approaches to the CKP. Also, we leave a heuristic algorithm for the general case with correlated weight values as the future research. In this case, properties of submodularity can no longer be used (see Atamt{\"u}rk and Bhardwaj \cite{atamturk2018network}).



\section*{References}
\bibliographystyle{elsarticle-num} 
\bibliography{ckp_ref}





\end{document}